\documentclass[11pt]{article}

\usepackage{graphicx, verbatim}
\usepackage{amssymb}
\usepackage{alltt}
\usepackage{amsmath, amssymb, amsfonts, amscd, xspace, pifont, amsthm}
\usepackage{mathrsfs}
\usepackage{algorithm}
\newtheorem{theorem}{Theorem}

\newtheorem{proposition}{Proposition}
\newtheorem{corollary}{Corollary}

\usepackage[round]{natbib}

\newcommand{\V}[1]{\ensuremath{\boldsymbol{#1}}\xspace}
\def\threeImages#1#2#3#4#5#6#7#8#9 
{
\centerline{\hfill\makebox[#2]{#3}\hfill\makebox[#5]{#6}\hfill\makebox[#8]{#9}\hfill}
\centerline{\hfill
\includegraphics[width=#2]{#1}
\hfill
\includegraphics[width=#5]{#4}
\hfill
\includegraphics[width=#8]{#7}
\hfill}
}

\def\twoImages#1#2#3#4#5#6 
{
\centerline{\hfill\makebox[#2]{#3}\hfill\makebox[#5]{#6}\hfill}
\centerline{\hfill
\includegraphics[width=#2]{#1}
\hfill
\includegraphics[width=#5]{#4}
\hfill}
}
\theoremstyle{remark}
\newtheorem{remark}{Remark}

\begin{document}
\title{A Note on New Bernstein-type Inequalities for the Log-likelihood Function of Bernoulli Variables}
\author{
	Yunpeng Zhao \thanks{School of Mathematical and Natural Sciences,
		Arizona State University, AZ, 85306. \texttt{Email:} yunpeng.zhao@asu.edu.} %%%Acknowlegements
	 } %%%Acknowledgements

\maketitle

\begin{abstract}
	We prove a new Bernstein-type inequality for the log-likelihood function of Bernoulli variables. In contrast to classical Bernstein's inequality and Hoeffding's inequality when applied to this log-likelihood, the new bound is independent of the parameters of the Bernoulli variables and therefore does not blow up as the parameters approach 0 or 1. The new inequality strengthens certain theoretical results on likelihood-based methods for community detection in networks and can be applied to other likelihood-based methods for binary data.

	\end{abstract}


\section{Introduction}
Let $X_1,X_2,...,X_n$ be  independent Bernoulli random variables, where $X_i$ takes the value 1 with probability $p_i$, denoted by Ber($p_{i}$). We are interested in deriving a concentration bound, which decays exponentially and is independent of parameters $p_i$, for the joint log-likelihood function of $X_1,X_2,...,X_n$. That is, 
\begin{align}
\mathbb{P} \left (   \left |   \sum_{i=1}^n \left (X_i \log p_i +(1-X_i) \log (1-p_i) \right ) -  \sum_{i=1}^n \left ( p_i \log p_i +(1-p_i) \log (1-p_i) \right) \right| \geq n \epsilon \right ) \leq c_1 e^{-c_2 n}, \label{motivating}
\end{align}
where $c_1$ and $c_2$ are constants that only depend on $\epsilon$. 

This research is motivated by theoretical studies of likelihood-based methods for binary data, in particular likelihood-based methods for community detection in networks. For example, Theorem 2 in \cite{choi2012stochastic} relies  on an inequality of this type and so does Theorem 2 in \cite{paul2016consistent}.  

We begin with classical results. By symmetry, we only consider
\begin{align*}
\mathbb{P} \left (   \left |   \sum_{i=1}^n X_i \log p_i -  \sum_{i=1}^n p_i \log p_i  \right|  \geq n \epsilon \right ). 
\end{align*}
Since $ X_i \log p_i \equiv p_i \log p_i$ almost surely when $p_i=1$ or 0 (using the convention $0\log 0=0$), the term can be dropped. Without loss of generality,  assume $p_i \in (0,1)$ for $i=1,...,n$. Noticing that $X_i \log p_i \in [\log p_{i} , 0]$ for $i=1,...,n$, we have Hoeffding's inequality \citep{Hoeffding63}: for all $\epsilon>0$, 
\begin{align*}
\mathbb{P} \left (   \left |  \sum_{i=1}^n (X_i-p_i)  \log p_i  \right| \geq  n \epsilon  \right ) \leq 2 \exp \left \{ -\frac{2 n^2 \epsilon^2}{ \sum_{i=1}^n (\log p_{i})^2 }  \right \}.
\end{align*}
Let $p_{(1)}$ be the smallest value among $p_1,...,p_n$. Then $|(X_i-p_i) \log p_i| \leq |\log p_{(1)}|$ for $i=1,...,n$. Bernstein's inequality (see \cite{dubhashi2009concentration}, Theorem 1.2) gives: for all $\epsilon>0$, 
\begin{align*}
\mathbb{P} \left (   \left |  \sum_{i=1}^n (X_i-p_i)  \log p_i  \right| \geq n \epsilon  \right ) \leq 2 \exp \left \{ -\frac{n^2 \epsilon^2/2 }{ \sum_{i=1}^n \textnormal{Var}(X_i\log p_i) + |\log p_{(1)} | n \epsilon/3}    \right \}.
\end{align*}

Note that both inequalities depend on $p_1,...,p_n$. As a result, when $p_{(1)}$ goes to 0 fast enough as $n$ grows, the bounds can be trivial due to the divergence of $|\log p_{(1)}|$. When applying these inequalities, technical assumptions are therefore needed to control the rate of the parameters going to the boundaries, for example, the condition on $P_{ij}$ in Theorem 2 of \cite{choi2012stochastic}. 

In this note, we prove a Bernstein-type inequality where the bound is independent of $p_1,...,p_n$. In other words, we show that $\sum_{i=1}^n (X_i-p_i)  \log p_i$ is in fact well-behaved when the parameters are near the boundary. The results such as in \cite{choi2012stochastic} and \cite{paul2016consistent} can therefore be strengthened by removing the technical assumptions. The new inequality is particularly useful in cases where those assumptions are not convenient to be made. 

\section{Main Result}\label{sec:main}
\begin{theorem}\label{main_thm}
	Let $X_i$ be independent $\textnormal{Ber}(p_i)$ for $i=1,...,n$ where $p_i \in [0,1]$. Let $G(p_i,\lambda)$ be the moment generating function (MGF) of $(X_i-p_i)  \log p_i$. Then for $|\lambda|<1$,
	\begin{align}
G(p_i,\lambda) \leq \exp \left \{ \frac{\lambda^2}{2(1-|\lambda|)}  \right  \}. \label{mgf_bound}
	\end{align}
	Furthermore, for all $t>0$,
	\begin{align}
	\mathbb{P} \left (   \left |  \sum_{i=1}^n (X_i-p_i)  \log p_i  \right| \geq t  \right ) \leq 2 \exp \left \{-\frac{t^2 }{2(n+t)}  \right \}. \label{main1}
	\end{align}
\end{theorem}
\begin{proof}
	Let $Y_i=(X_i-p_i)  \log p_i$. 
Then 
	\begin{align*}
	G(p_i,\lambda)=\mathbb{E} [e^{\lambda Y_i}]= p_i e^{\lambda(1-p_i)\log p_i}+(1-p_i) e^{-\lambda p_i \log p_i}. 
	\end{align*}
	 The key step is to prove \eqref{mgf_bound}, which is an exponential upper bound for $G(p_i,\lambda)$. Note that the upper bound is independent of $p_i$. First consider the case where $p_i \in (0,1)$.

	We prove a Bernstein's condition (see \cite{wainwright_2019}, p. 27 for an introduction) for the moments of $Y_i$. That is, find constants $\sigma^2$ and $b$, such that 
	\begin{align}
	| \mathbb{E} [  Y_i^m ] | \leq \frac{1}{2} m! \sigma^2 b^{m-2} \quad \textnormal{for } m=3,4,... \label{Bern_cond}.
	\end{align}
	Different from \cite{wainwright_2019}, here we look for  constants $\sigma^2$ and $b$ which are independent of $p_i$.
	
	Consider
\begin{align*}
\mathbb{E} [  Y_i^m ] = \underbrace{p_i (1-p_i)^m (\log p_i)^m}_{A_1} +\underbrace{(1-p_i) (-p_i\log p_i)^m}_{A_2}.
\end{align*}
By taking the first and the second derivatives of $p_i(\log p_i)^m$, one can easily check that its optimum is achieved at $p_i=e^{-m}$. Therefore, 
\begin{align*}
|A_1| \leq | p_i (\log p_i)^m | \leq \left (  \frac{m}{e}\right )^m \leq \frac{m!}{\sqrt{2\pi m}},
\end{align*}
where the last inequality follows from Stirling's formula \citep{robbins1955remark}. 
Similarly,
\begin{align*}
|A_2| \leq (1-p_i) (-p_i\log p_i)^m \leq e^{-m}.
\end{align*}
It follows that
\begin{align*}
| \mathbb{E} [  Y_i^m ] | \leq \frac{m!}{\sqrt{2\pi m}}+\frac{1}{e^m} \leq \frac{1}{2} m! \quad \textnormal{for } m=3,4,... .
\end{align*}
Therefore, the Bernstein's condition \eqref{Bern_cond} holds when $\sigma^2=1$ and $b=1$. 

We now use the Bernstein's condition to prove \eqref{mgf_bound}. The argument is similar to \cite{wainwright_2019}, pp. 27-28. We give the details for completeness.

 By the power series expansion of the exponential function and Fubini's theorem (for exchanging the expectation and summation due to $\mathbb{E}[e^{|\lambda Y_i|}]<\infty$), 
\begin{align*}
G(p_i,\lambda)=\mathbb{E} [e^{\lambda Y_i}] & = 1+ \frac{\lambda^2 \textnormal{Var}(Y_i)}{2}+\sum_{m=3}^{\infty} \lambda^m \frac{\mathbb{E} [  Y_i^m ]}{m!} \\
& \leq 1+ \frac{\lambda^2}{2}+\frac{\lambda^2}{2}  \sum_{m=1}^{\infty} |\lambda|^m, 
\end{align*}
where the inequality follows from  the Bernstein's condition \eqref{Bern_cond} and $\textnormal{Var}(Y_i)=p_i(1-p_i) (\log p_i)^2 \leq 1$. For any $|\lambda|<1$, the geometrics series converges, and
\begin{align}
G(p_i,\lambda) \leq 1+ \frac{\lambda^2}{2}\frac{1}{1-|\lambda|}  \leq \exp \left \{ \frac{\lambda^2}{2(1-|\lambda|)}  \right  \}, 
\end{align}
where the second inequality follows from $1+s\leq e^s$. Notice that $G(p_i,\lambda)\equiv 1$ for $p_i=0$ or $1$ so the inequality holds for all $p_i\in [0,1]$.   

The rest of the proof follows from a standard argument using the Chernoff bound, which can be found in a standard textbook on concentration inequalities, for example, \cite{dubhashi2009concentration}, Chapter 1. We give the details for readers who are unfamiliar with this technique. For $-1<\lambda<0$, 
\begin{align*}
	& \mathbb{P} \left (     \sum_{i=1}^n Y_i  \leq -t  \right ) = \mathbb{P} \left ( e^{\lambda \sum_{i=1}^n Y_i}  \geq e^{-\lambda t}  \right ) \leq \frac{\prod_{i=1}^n \mathbb{E} [e^{\lambda Y_i}]}{ e^{-\lambda t} } \leq \exp \left \{ \frac{n\lambda^2}{2(1-|\lambda|)} +\lambda t \right  \}, \\
\end{align*}
where the first inequality is Markov's inequality and the second inequality follows from \eqref{mgf_bound}. By setting $\lambda=-\frac{t}{t+n} \in (-1,0)$, we obtain
\begin{align*}
& \mathbb{P} \left (     \sum_{i=1}^n Y_i  \leq -t  \right ) \leq \exp \left \{-\frac{t^2 }{2(n+t)}  \right \}. 
\end{align*}
The bound for the right tail can be obtained similarly by setting $\lambda=\frac{t}{t+n}$. 
	\end{proof}

%Before proceeding, we add several remarks.
\begin{remark}
$\mathbb{E} [  Y_i^m ]$ is dominated by the term $p_i (1-p_i)^m (\log p_i)^m$, which has a bump near the boundary, -- that is, its value achieves the order of $(m/e)^m$ at $p_i=e^{-m}$.  This value is, however, still bounded by $m!$, which implies the left-tail bound of $\sum_{i=1}^n Y_i$ is well-behaved when the parameters are near the boundary. 	
\end{remark}
\begin{remark}
The constant $\sigma^2=1$ in the Bernstein's condition \eqref{Bern_cond} is not the optimal value. We simply choose this value for obtaining a nice form in  \eqref{main1}. The constant $b=1$ is optimal because $1/\sqrt{2\pi m}$ dominates $b^{m-2}$ for any $0<b<1$. This fact can also be seen from the following proposition: 
\end{remark}
\begin{proposition}
For $\lambda<-1$, $ \lim_{p \rightarrow 0^+} G(p,\lambda)=\infty $, which implies $G(p,\lambda)$ cannot be bounded by any function that takes finite values. For $\lambda \geq -1$, $\lim_{p \rightarrow 0^+} G(p,\lambda)<\infty$. 
\end{proposition}
\begin{proof}
	The result is obvious by noticing that $p e^{\lambda \log p}=p^{\lambda+1}$. 
\end{proof}

We now prove \eqref{motivating}. We state a slightly more general result for multinoulli variables. Let $\V{X}_i=(X_{i1},...,X_{iK})$ be a multinoulli variable with $p_{ik}=\mathbb{P}(X_{ik}=1)$, and assume $\V{X}_1,...,\V{X}_n$ are independent. 
\begin{corollary}\label{cl1}
	For $p_{ik}\in [0,1]$, $i=1,...,n$, $k=1,...,K$, and all $\epsilon>0$,
	\begin{align*}
	\mathbb{P} \left (   \left |   \sum_{i=1}^n \sum_{k=1}^K   (X_{ik}-p_{ik}) \log p_{ik}   \right| \geq n \epsilon \right ) \leq 2K \exp \left \{ -\frac{n \epsilon^2}{2K(K+\epsilon)} \right \}.
	\end{align*}
\end{corollary}
\begin{proof}
	The result is obvious by noticing that 
\begin{align*}
	\mathbb{P} \left (   \left |   \sum_{i=1}^n \sum_{k=1}^K   (X_{ik}-p_{ik}) \log p_{ik}   \right| \geq n \epsilon \right ) \leq \sum_{k=1}^K \mathbb{P} \left (   \left |  \sum_{i=1}^n  (X_{ik}-p_{ik}) \log p_{ik}   \right| \geq \frac{n \epsilon}{K} \right ) ,
\end{align*}
and setting $t= n \epsilon /K$ in \eqref{main1}.
\end{proof}

\section{Extension to Grouped Observations}
We now extend our result to a setup where the observations are grouped into different classes. In fact, this is the setup that can be directly applied to the community detection literature, for example, Theorem 2 in \cite{choi2012stochastic} and Theorem 2 in \cite{paul2016consistent}. We will also apply the result in a working paper by the author and collaborators on the theory of hub models, a special latent class model for binary data proposed by \cite{zhao2019network}. 

Let 
$X_{1}^{(1)},X_{2}^{(1)},...,X_{n_1}^{(1)},X_{1}^{(2)},X_{2}^{(2)},...,X_{n_2}^{(2)},...,X_{1}^{(I)},X_{2}^{(I)},...,X_{n_I}^{(I)}$ be independent Bernoulli variables, where $p_{j}^{(i)}$ is the parameter for $X_{j}^{(i)}$. Let $\sum_{i=1}^I n_i=n$. And let $\bar{p}^{(i)}= \frac{1}{n_i} \sum_{j=1}^{n_i}p_{j}^{(i)}$ for $i=1,...,I$, where $\bar{p}^{(i)} \in [0,1]$. 

\begin{theorem}
	For all $t>0$, 
	\begin{align}\label{extension}
%	\mathbb{P}\left (  \left | \sum_{i=1}^I   \left ( \sum_{j=1}^{n_i} X_{j}^{(i)}  \log \bar{p}^{(i)} -n_i \bar{p}^{(i)}\log \bar{p}^{(i)} \right ) \right | \geq t  \right ) \leq 2\exp \left \{-\frac{t^2 }{2(n+t)} \right \}. \\
		\mathbb{P}\left (  \left | \sum_{i=1}^I  \sum_{j=1}^{n_i}  ( X_{j}^{(i)}-p_j^{(i)} ) \log \bar{p}^{(i)}  \right | \geq t  \right ) \leq 2\exp \left \{-\frac{t^2 }{2(n+t)} \right \}.
	\end{align}	
\end{theorem}
Note that here the model assumption on $X^{(i)}_{j}$ is identical to the setup in Section \ref{sec:main}, where each Bernoulli variable has its own parameter. The function we consider in the inequality is, however, defined differently. Moreover, this theorem reduces to Theorem \ref{main_thm} when $n_i\equiv 1$ for $i=1,...,I$. 
\begin{proof}
	Let $Z^{(i)}=  \sum_{j=1}^{n_i}  ( X_{j}^{(i)}-p_j^{(i)} ) \log \bar{p}^{(i)}= \sum_{j=1}^{n_i}  ( X_{j}^{(i)}-\bar{p}^{(i)} ) \log \bar{p}^{(i)}$. Consider the MGF $\mathbb{E} [e^{\lambda Z^{(i)}}] $ for $\bar{p}^{(i)} \in (0,1)$.
	\begin{align*}
	\mathbb{E} [e^{\lambda Z^{(i)}}]= & \prod_{j=1}^{n_i} \left ( p_{j}^{(i)} e^{\lambda(1-\bar{p}^{(i)})\log \bar{p}^{(i)}}+(1-p_{j}^{(i)}) e^{-\lambda \bar{p}^{(i)} \log \bar{p}^{(i)}} \right ) \\
	\leq & \left ( \bar{p}^{(i)} e^{\lambda(1-\bar{p}^{(i)})\log \bar{p}^{(i)}}+(1-\bar{p}^{(i)}) e^{-\lambda \bar{p}^{(i)} \log \bar{p}^{(i)}} \right )^{n_i}=(G(\bar{p}^{(i)},\lambda))^{n_i},
	\end{align*}
	where the inequality follows from the inequality of arithmetic and geometric means: $ \sqrt[n]{\prod_{i=1}^n a_i} \leq \sum_{i=1}^n a_i/ n$ for non-negative $a_1,...,a_n$. From \eqref{mgf_bound}, $G(\bar{p}^{(i)},\lambda) \leq  \exp \left \{ \frac{\lambda^2}{2(1-|\lambda|)}  \right  \}$ for $|\lambda|<1$. It follows that $	\mathbb{E} [e^{\lambda Z^{(i)}}]\leq \exp \left \{ \frac{n_i \lambda^2}{2(1-|\lambda|)}  \right  \} $ for $|\lambda|<1$. The inequality also holds for $\bar{p}^{(i)}=0$ or 1 as $\mathbb{E} [e^{\lambda Z^{(i)}}] \equiv 1$. The rest of the proof follows from the standard argument using the Chernoff bound as shown in the proof of Theorem \ref{main_thm}.
\end{proof}

We conclude this note with a corollary that is easily proved  by the same argument for Corollary \ref{cl1}. Let 
$\V{X}_{1}^{(1)},\V{X}_{2}^{(1)},...,\V{X}_{n_1}^{(1)},\V{X}_{1}^{(2)},\V{X}_{2}^{(2)},...,\V{X}_{n_2}^{(2)},...,\V{X}_{1}^{(I)},\V{X}_{2}^{(I)},...,\V{X}_{n_I}^{(I)}$ be independent multinoulli variables, where each $\V{X}_{j}^{(i)}=\left (X_{j1}^{(i)},...,X_{jK}^{(i)} \right )$, and $p_{jk}^{(i)}=\mathbb{P}(X_{jk}^{(i)}=1)$ for $k=1,...,K$. As before, let $n=\sum_{i=1}^I n_i$. And let $\bar{p}_k^{(i)}= \frac{1}{n_i} \sum_{j=1}^{n_i}p_{jk}^{(i)}$ for $i=1,...,I$ and $k=1,...K$, where $\bar{p}_k^{(i)} \in [0,1]$. 

\begin{corollary}
For all $\epsilon>0$, 
	\begin{align*}
\mathbb{P}\left (  \left |  \sum_{i=1}^I   \sum_{j=1}^{n_i} \sum_{k=1}^K( X_{jk}^{(i)}   -p_{jk}^{(i)} )\log \bar{p}^{(i)}_k  \right | \geq n \epsilon \right  ) \leq 2K \exp \left \{ -\frac{n \epsilon^2}{2K(K+\epsilon)} \right \}.
	\end{align*}
\end{corollary}

\end{document}